\documentclass[11pt]{amsart}
\usepackage{amsmath,amssymb,amscd,mathrsfs}
\usepackage{eepic,epic}
\usepackage{longtable}
\newtheorem{introth}{Theorem}

\newtheorem{introdefinition}[introth]{Definition}

\newtheorem{introremark}[introth]{Remark}
\newenvironment{intrormk}{\begin{introremark}\rm}{\end{introremark}}

\newtheorem{thm}{Theorem}[section]
\newtheorem{prop}[thm]{Proposition}
\newtheorem{lem}[thm]{Lemma}
\newtheorem{cor}[thm]{Corollary}
\newtheorem{definition}[thm]{Definition} 
\newtheorem{remark}[thm]{Remark}
\newtheorem{observation}[thm]{Observation}
\newtheorem{example}[thm]{Example}

\newenvironment{rmk}{\begin{remark}\rm}{\end{remark}}

  \def\R{\mathbb R} 
\def\P{\mathbb P} 

\def\F{\mathcal F}

\newenvironment{(enumerate)}{
  \begin{enumerate}
  
  }{\end{enumerate}}

\begin{document} 
\title[Automorphism groups of Enriques surfaces]{The automorphism groups of 
Enriques surfaces covered by symmetric quartic surfaces}
\author[S. Mukai]{Shigeru Mukai}
\author[H. Ohashi]{Hisanori Ohashi}
\address{Research Institute for Mathematical Sciences,
Kyoto University,
Kyoto 606-8502,
JAPAN }
\email{mukai@kurims.kyoto-u.ac.jp}
\address{Department of Mathematics, 
Faculty of Science and Technology, 
Tokyo University of Science, 
2641 Yamazaki, Noda, 
Chiba 278-8510, JAPAN}
\email{ohashi@ma.noda.tus.ac.jp, ohashi.hisanori@gmail.com}
\thanks{Supported in part by 
the JSPS Grant-in-Aid for Scientific Research (B) 22340007, (S) 19104001, (S) 22224001, (S)25220701, (A) 22244003, for Exploratory Research 20654004 and for Young Scientists (B) 23740010.}


\maketitle

\begin{center}
Dedicated to Prof. Robert Lazarsfeld on his 60th birthday
\end{center}

\begin{abstract} 
Let  $S$  be the (minimal) Enriques surface 
obtained from the symmetric quartic surface
$(\sum_{i<j}x_ix_j)^2=kx_1x_2x_3x_4$  in $\P^3$ with $k\neq 0,4,36$,
by taking quotient of the Cremona action $(x_i) \mapsto (1/x_i)$.
The automorphism group of $S$ is a semi-direct product 
of a free product $\F$ of four involutions and 
the symmetric group $\mathfrak{S}_4$. Up to action of $\F$, there are 
exactly $29$ elliptic pencils on $S$. 
\end{abstract}
\date{\today} 

The automorphism groups of very general  Enriques surfaces, namely those corresponding to very general points in moduli, were computed in Barth-Peters\cite{BP}
as an explicitly described infinite arithmetic group. 
Also many authors \cite{D,BP,N,kondo86} studied 
Enriques surfaces with only finitely many automorphisms. 
The article \cite{BP} also includes an example whose automorphism group is infinite but still virtually abelian group.
In this paper we give a concrete example of an Enriques surface whose 
automorphism group is not virtually abelian.
Moreover, the automorphism group is explicitly described in terms of generators and relations.
See also Remark~\ref{vcd}.

We work over any algebraically closed field whose characteristic is not two.
Let us introduce the quartic surface with parameters $k$ and $l$,
\begin{equation}\label{the net}
\overline{X}\colon \{ s_2^2=ks_4+ls_1s_3 \}\subset \P^3,
\end{equation}
where $s_d$ are the fundamental symmetric polynomials of degree $d$ in the homogeneous 
coordinates $x_1,\dots,x_4$. It is singular at the four coordinate points $(1:0:0:0) , \dots,
(0:0:0:1)$ and has an action of the symmetric 
group $\mathfrak{S}_4$. It also admits the action of the standard Cremona 
transformation
\[\varepsilon \colon (x_1:\cdots:x_4)\mapsto \left(\frac{1}{x_1}:\cdots:\frac{1}{x_4}\right)\]
which commutes with $\mathfrak{S}_4$. After taking the minimal resolution $X$,
the quotient surface $S=X/\varepsilon$ becomes an Enriques surface, whenever $\overline{X}$ avoids the eight fixed points 
$(\pm 1:\pm 1:\pm 1:1)$ of $\varepsilon$. This condition is equivalent to 
$k+16l\neq36, k\neq 4$ and $4l+k\neq 0.$

The projection from one of four coordinate points exhibits $X$ as a double cover of the 
projective plane $\P^2$. The associated covering involution
commutes with $\varepsilon$ and defines an involution of the Enriques surface $S$.
In this way we obtain four involutions  
$\sigma_i \ (i=1,\dots,4)$. 
The action of $\mathfrak{S}_4$ also descends to $S$. Therefore,
by mapping the generators of $C_2^{*4}$ to $\sigma_i$, 
we obtain a group homomorphism 
\begin{equation}\label{the hom}
\mathfrak{S}_4 \ltimes (C_2^{*4})\rightarrow \mathrm{Aut}(S),
\end{equation}
where $\mathfrak{S}_4$  acts on the free product as permutation of the four factors.

In this paper we study the automorphism group and elliptic fibrations of $S$ in the 
case $l=0$. Our main result is as follows.
\begin{introth}\label{MT}{\rm (=Theorem~\ref{mt})}
In the equation \eqref{the net}, let $l=0$ and $k\neq 0,4,36$.
Then \eqref{the hom} is an isomorphism. 
Namely $\mathrm{Aut}(S)$ is isomorphic to the semi-direct product 
of the free product $\F$ of four involutions $\sigma_i\ (i=1,\dots,4)$ and 
the symmetric group $\mathfrak{S}_4$. 
\end{introth}
In the proof of this theorem, we also obtain the following results on elliptic pencils and 
smooth rational curves. Let $S$ be as in Theorem \ref{MT}.
\begin{introth}\label{elliptic} {\rm (= Theorem~\ref{ellip})}
Up to the action of the free product $\F\simeq C_2^{*4}$, there are exactly 
$29$ elliptic pencils on $S$. They are classified into five types and the
main properties are as in the following table. 

\begin{longtable}{c|c|c|c}
    & singular fibers & Mordell-Weil rank & number \\ \hline
$1)$ & $\tilde{E_7}+\tilde{A_1}$ & $0$ & $12$ \\
$2)$ & $\tilde{E_6}+\tilde{A_2}$ & $0$ & $4$\\
$3)$ & $\tilde{D_6}+\tilde{A_1}$ & $1$ & $6$\\
$4)$ & $\tilde{A_7}+\tilde{A_1}$ & $0$ & $3$\\
$5)$ & $2\tilde{A_5}+\tilde{A_2}+\tilde{A_1}$ & $0$ & $4$ 
\end{longtable}\noindent
Here $2\tilde{A_5}$ denotes the multiple fiber and
the Mordell-Weil rank stands for that of 
its Jacobian fibration. 
\end{introth}

\begin{introth}\label{rational} {\rm (= Theorem~\ref{rat})}
Up to the action of the free product 
$\F\simeq C_2^{*4}$, there are exactly sixteen smooth rational curves on $S$.
They are represented by the curves in the configuration $10A+6B$ (see below).
\end{introth}

The proof of Theorem \ref{MT} uses some sixteen smooth rational curves on $S$ and 
the fact that four involutions $\sigma_1,\dots,\sigma_4$ are numerically reflective.
First using the four singularities of type $D_4$ and four tropes on $\overline{X}$, we find
ten smooth rational curves on $S$ with the dual graph as in Figure \ref{10A} 
(Section \ref{config}).
We call it the $10A$ configuration.

\begin{figure}
\centering
\begin{picture}(80,90)
\linethickness{0.7pt}
\put(0,0){\circle*{6}}
\put(-16,3){$E_2$}
\put(30,0){\circle*{6}}
\put(60,0){\circle*{6}}
\put(64,-6){$E_3$}
\put(30,52){\circle*{6}}
\put(28,56){$E_1$}
\put(15,26){\circle*{6}}
\put(-1,30){$E_{12}$}
\put(45,26){\circle*{6}}
\put(72,24){\circle*{6}}
\put(77,24){$E_{4}$}
\put(66,12){\circle*{6}}
\put(51,38){\circle*{6}}
\put(55,38){$E_{14}$}
\put(36,12){\circle*{6}}
\put(0,0){\line(1,0){60}}
\put(60,0){\line(-15,26){30}}
\put(30,52){\line(-15,-26){30}}
\put(30,52){\line(3,-2){40}}
\put(72,24){\line(-1,-2){12}}
\multiput(0,0)(6,2){12}{\line(3,1){3}}
\end{picture}
\caption{The $10A$ configuration.}
\label{10A}
\end{figure}
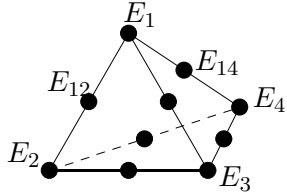

Also by looking at some other plane sections, we find further six smooth rational curves
on $S$ with the dual graph as in Figure \ref{6B}. This is called the $6B$ configuration.

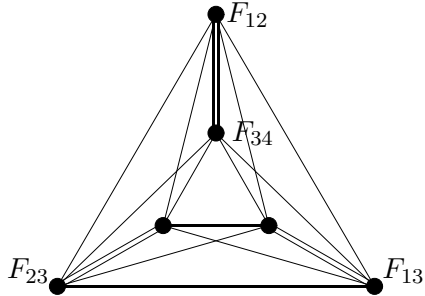
\begin{figure}
\centering
\begin{picture}(140,120)
\linethickness{0.7pt}
\put(10,0){\circle*{6}}	\put(-9,3){$F_{23}$}
\put(130,0){\circle*{6}}	\put(133,3){$F_{13}$}
\put(70,103){\circle*{6}}\put(74,100){$F_{12}$}
\put(50,23){\circle*{6}}
\put(90,23){\circle*{6}}
\put(70,58){\circle*{6}}	\put(76,55){$F_{34}$}
\put(10,0){\line(1,0){120}}
\put(10,0){\line(60,103){60}}
\put(130,0){\line(-60,103){60}}
\put(50,23){\line(1,0){40}}
\put(50,23){\line(20,35){20}}
\put(90,23){\line(-20,35){20}}
\put(10,0){\line(60,58){60}}
\put(10,0){\line(80,23){80}}
\put(130,0){\line(-60,58){60}}
\put(130,0){\line(-80,23){80}}
\put(70,103){\line(-20,-80){20}}
\put(70,103){\line(20,-80){20}}
\put(69,103){\line(0,-1){45}}
\put(71,103){\line(0,-1){45}}
\put(10.5,-0.85){\line(40,23){40}}
\put(9.5,0.85){\line(40,23){40}}
\put(129.5,-0.85){\line(-40,23){40}}
\put(130.5,0.85){\line(-40,23){40}}
\end{picture}
\caption{The $6B$ configuration}
\label{6B}
\end{figure}

We denote by $NS(S)_f$ the N\'{e}ron-Severi lattice of $S$ modulo torsion.
The action of involutions $\sigma_1,\dots,\sigma_4$ on $NS(S)_f$ 
is the reflection in $(-2)$ classes $G_1,\dots,G_4\in NS(S)_f$, respectively. 
For instance, the class $G_1$ is $E_2+E_{23}+E_3+E_{34}+E_4+E_{24}-E_1$ in terms of Figure \ref{10A} (Proposition~\ref{NR}).  
The dual graph of these four $(-2)$ classes 
is the complete graph in four vertices with doubled edges. 
This is what is called the $4C$ configuration.

We can check that the twenty $(-2)$-classes $E_i, E_{ij}, F_{ij}, G_i$ 
define a convex polyhedron whose Coxeter diagram satisfies the Vinberg's condition 
\cite{some}. Namely the subgroup $W(10A+6B+4C)$ 
generated by reflections in these twenty classes
has finite index in the orthogonal group $O(NS(S)_f)$.
In fact, the limit of our Enriques surfaces as  $k \to \infty$  is of type V in Kondo\cite{kondo86} (see Remark~\ref{limit}), and our diagram coincides with Kondo's.
Although in his case the classes $G_1,\dots,G_4$ were also represented by 
smooth rational curves, in our case they appear just as the {\em{center}} of the reflective 
involutions $\sigma_1,\dots,\sigma_4$ and are not effective (Corollary \ref{G_i}).

To prove our Theorem \ref{MT}, we divide the generators of $W(10A+6B+4C)$ into two parts,
those coming from $10A+6B$ and those from $4C$. 
By a lemma of Vinberg\cite{Vinberg}, $W(10A+6B+4C)$ is the semi-direct product
$W(4C) \ltimes \overline{N}(W(10A+6B))$, where $\overline{N}$ denotes the normal closure.
Since the whole $10A+6B+4C$ configuration has only $\mathfrak{S}_4$-symmetry, we obtain our Theorem~\ref{MT} and the others (Sections \ref{proof1}). 

\begin{intrormk} There are some interesting cases in $l\neq 0$, too.

(1) When $(k-4)(l-4)=16$, the surface $\overline{X}$ is 
Kummer's quartic surface $\mathrm{Km}(J(C))$ written in Hutchinson's form. It has $16$ nodes.
Our four involutions $\sigma_i$ are called {\em{projections}}. As is shown in 
\cite{mukai-ref}, $S$ is an Enriques surface of Hutchinson-G\"{o}pel type and 
the four involutions are numerically reflective.
Especially in the case $(k,l)=(-4,2)$, the hyperelliptic curve $C$ branches over 
the vertices of regular octahedron and the equation of $\overline{X}$ becomes 
\[(x_1^2x_2^2+x_3^2x_4^2)+(x_1^2x_3^2+x_2^2x_4^2)+(x_1^2x_4^2+x_2^2x_3^2)+2x_1x_2x_3x_4=0.\]
This is the case of the octahedral Enriques surface \cite{Fields} and $S$ is isomorphic to the 
normalization of the singular sextic surface 
\[x_1^2+x_2^2+x_3^2+x_4^2+\sqrt{-1}\left(\frac{1}{x_1^2}+\frac{1}{x_2^2}+\frac{1}{x_3^2}+\frac{1}{x_4^2}\right)x_1x_2x_3x_4=0.\]
In these cases we know that there exist automorphisms on $S$ induced from 
$X$ other than projections, namely some switches and correlations. 
The automorhism group of the octahedral Enriques surface will be discussed elsewhere.

(2)  The quartic surface $\overline{X}: ks_4+ls_1s_3=0$ is the Hessian of the 
cubic surface 
\[k(x_1^3+x_2^3+x_3^3+x_4^3)+l(x_1+x_2+x_3+x_4)^3=0.\]
The case $(k:l)=(1:-1)$ is most symmetric among this one-parameter family.
In this special case, the Enriques surface $S = X/\varepsilon$  is of type VI in Kondo\cite{kondo86} and
the automorphism group is isomorphic to $\mathfrak{S}_5$. 
In particular, the homomorphism \eqref{the hom} is neither injective nor surjective.
\end{intrormk}

\begin{intrormk}\label{vcd}
In terms of virtual cohomological dimensions of discrete groups \cite{Serre}, 
our example can be located in the following way. The virtual cohomological dimension
is equal to $0$ for finite groups. On the other extreme,
the discrete group $\mathrm{Aut} (S)$ for very general Enriques surfaces $S$ has 
the virtual cohomological dimension $8$. See \cite{BS}. In our case, the automorphism group 
has virtual cohomological dimension $1$.
\end{intrormk}

\section{Smooth rational curves}\label{config}

Under the condition $l=0$, the equation \eqref{the net} becomes
\begin{equation}\label{l=0}
\overline{X}\colon (x_1x_2+x_1x_3+x_1x_4+x_2x_3+x_2x_4+x_3x_4)^2=kx_1x_2x_3x_4.
\end{equation}
This surface has four rational double points of type $D_4$ at the four 
coordinate points $(1:0:0:0),\dots,(0:0:0:1)$ and by taking 
the quotient of the minimal resolution $X$ by the standard Cremona 
involution 
\[ \varepsilon\colon  (x_1:\cdots:x_4)\mapsto \left(\frac{1}{x_1}:\cdots:\frac{1}{x_4}\right),\]
we obtain an Enriques surface $S=X/\varepsilon$.
We begin with the study of the configuration of smooth rational curves 
on the surfaces.

The desingularization $X$ has sixteen smooth rational curves as the exceptional curves of the 
four $D_4$ singularities.
Also each coordinate plane cuts the quartic doubly along a conic, 
which defines a smooth rational curve on $X$. They are called {\em{tropes}}.
The configuration of these twenty curves is as in
Figure \ref{4D4}, which depicts the dual graph. Black vertices come from the singularities 
and white ones are tropes. 

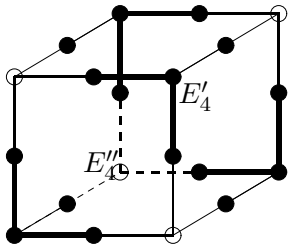
\begin{figure}
\centering
\begin{picture}(120,90)
\linethickness{0.7pt}
\put(0,0){\line(1,0){60}}
\put(60,0){\line(0,1){60}}
\put(60,60){\line(-1,0){60}}
\put(0,60){\line(0,-1){60}}
\put(0,60){\line(5,3){40}}
\put(40,84){\line(1,0){60}}
\put(100,84){\line(-5,-3){40}}
\put(100,84){\line(0,-1){60}}
\put(100,24){\line(-5,-3){40}}
\multiput(0,0)(4,2.4){10}{\line(5,3){2}}
\multiput(40,24)(0,6){10}{\line(0,1){3}}
\multiput(40,24)(6,0){10}{\line(1,0){3}}
\put(0,60){\circle{6}}
\put(60,0){\circle{6}}
\put(40,24){\circle{6}}
\put(26,24){$E_4''$}
\put(100,84){\circle{6}}
\linethickness{1.7pt}
\put(0,0){\circle*{6}}
\put(30,0){\circle*{6}}
\put(0,30){\circle*{6}}
\put(20,12){\circle*{6}}
\put(0,0){\line(0,1){30}}
\put(0,0){\line(1,0){30}}
\put(0,0){\line(5,3){20}}
\put(60,60){\circle*{6}}
\put(62,50){$E_4'$}
\put(60,30){\circle*{6}}
\put(30,60){\circle*{6}}
\put(80,72){\circle*{6}}
\put(60,60){\line(0,-1){30}}
\put(60,60){\line(-1,0){30}}
\put(60,60){\line(5,3){20}}
\put(40,84){\circle*{6}}
\put(40,54){\circle*{6}}
\put(70,84){\circle*{6}}
\put(20,72){\circle*{6}}
\put(40,84){\line(0,-1){30}}
\put(40,84){\line(1,0){30}}
\put(40,84){\line(-5,-3){20}}
\put(100,24){\circle*{6}}
\put(100,54){\circle*{6}}
\put(70,24){\circle*{6}}
\put(80,12){\circle*{6}}
\put(100,24){\line(0,1){30}}
\put(100,24){\line(-1,0){30}}
\put(100,24){\line(-5,-3){20}}

\end{picture}
\caption{The quartic surface with four $D_4$ singularities.}
\label{4D4}
\end{figure}

The standard Cremona involution $\varepsilon$ acts on Figure \ref{4D4}
by the point symmetry.
Therefore the Enriques surface $S$ has 
ten smooth rational curves whose dual graph 
is the one in Figure \ref{10A}.
In what follows, we call these ten curves on $S$ the $10A$ configuration. 
The indexing is given as follows. 
Since a 
vertex of the tetrahedron corresponds to two curves on $X$, namely the trope $\{x_i=0\}$
and the central component of the exceptional curves at $(0:\cdots :1: \cdots :0)$ (the $i$-th 
coordinate is $1$), 
we denote the curve at the vertex by $E_i$ $(i=1,\dots,4)$. 
Also if a vertex at the middle of an edge is connected to
two vertices, say $E_i$ and $E_j$, then we denote the curve by $E_{ij}$. This is the first 
configuration of smooth rational curves on $S$ of our interest.
It is convenient to note that the ten curves $\{E_i,E_{ij}\}$ generate $NS(S)_f$ over the rationals;
the Gram matrix of these curves has determinant $-64$.\\

Next let us consider the six plane sections by $\{x_i+x_j=0\}$ $(i=1,\dots,4)$. 
In the equation \eqref{l=0}, we see that each plane section decomposes into two conics which
are disjoint on $X$ and exchanged by $\varepsilon$. 
Thus we obtain further six smooth rational curves on $S$, naturally indexed as $F_{ij}$.
The intersection relation between these curves is as in Figure \ref{6B}. We call it 
the $6B$ configuration.
Moreover,
we can clarify the intersection relations between the configurations as follows.
\begin{equation*}
(E_k,F_{ij})=0;\qquad
(E_{kl},F_{ij})=
\begin{cases}
2\ \text{if $\{k,l\} =\{i,j\}$},\\
0\ \text{otherwise}.
\end{cases}
\end{equation*}
The configuration of sixteen curves thus obtained is denoted by $10A+6B$.

\section{Numerically reflective involutions}\label{involutions}

The quartic surface \eqref{l=0} 
can be exhibited as a double cover of $\P^2$ 
by the projection from one of the coordinate points, say $(0:0:0:1)$. 
The branch $B\subset \P^2$ is the sextic plane curve defined by 
\begin{equation}\label{2.2}
x_1x_2x_3\left(4(x_1+x_2+x_3)\left(\frac{1}{x_1}+\frac{1}{x_2}+\frac{1}{x_3}\right)x_1x_2x_3-kx_1x_2x_3\right) =0.
\end{equation}
It is the union of the coordinate triangle $\{x_1x_2x_3=0\}$
and the cubic curve 
\begin{equation}\label{C}
C\colon 4(x_1+x_2+x_3)\left(\frac{1}{x_1}+\frac{1}{x_2}+\frac{1}{x_3}\right)-k=0,
\end{equation}
which is invariant under the Cremona transformation $(x_i)\mapsto (1/x_i)$ of $\P^2$.
See Figure \ref{branch}.
In this double plane picture, the twenty rational curves in Figure \ref{4D4} 
can be seen as 
the twelve rational curves above the three triple points of $B$, three rational curves above 
the three 
nodes of $B$, three tropes as the inverse image of the coordinate triangle and some components of 
inverse images of the curves $L\colon \{x_1+x_2+x_3=0\}$ and 
$Q\colon \{\frac{1}{x_1}+\frac{1}{x_2}+\frac{1}{x_3}=0\}.$ (We note that
the line $L$ must pass through the three simple intersection points of $C$ with 
the triangle in Figure \ref{branch}, although it is not visible.)

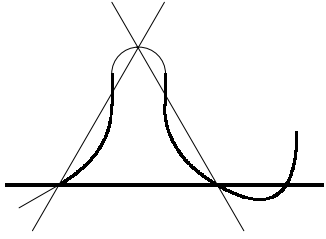
\begin{figure}
\centering
\begin{picture}(140,100)
\linethickness{1.2pt}
\put(0,20){\line(1,0){120}}
\put(10,2.7){\line(30,52){50}}
\put(90,2.7){\line(-30,52){50}}
\linethickness{0.7pt}
\put(5,11.35){\line(45,26){15}}
\qbezier(20,20)(40,31.6)(40,50)
\put(50,50){\oval(20,44)[t]}
\qbezier(60,50)(60,31.6)(80,20)
\qbezier(80,20)(110,2.7)(110,40)
\end{picture}
\caption{The branch sextic $B$}
\label{branch}
\end{figure}

The covering transformation of this double cover $X\rightarrow \P^2$ is called the 
{\em{projection}}. It is an anti-symplectic involution acting on $X$.
It stabilizes all the curves above the branch curve $B$ (including the ones above the
singularities of $B$).
In particular, in Figure \ref{4D4}, if $E_4''$ is the trope $\{x_4=0\}$,  
then the projection stabilizes all the curves 
except for $E_4''$ and its antipodal $E_4'$ (coming from the singularity at $(0:0:0:1)$).
It is easy to determine the fixed curves of the projection and
it consists of 
six rational curves (vertices of the cube except for $E_4'$ and $E_4''$) and 
the inverse image of the elliptic curve $C$.

Since the projection commutes with the Cremona involution of $\P^3$, we obtain an 
involution of the Enriques surface $S$. It is denoted by $\sigma_4$, where the index
is in accordance to the center of the projection $(0:0:0:1)$.
\begin{prop}\label{NR}
The involution $\sigma_4\in \mathrm{Aut}(S)$ is numerically reflective.
Moreover, its action on the N\'{e}ron-Severi lattice $NS(S)_f$ 
is the reflection in the divisor $G_4=E_1+E_{12}+E_2+E_{23}+E_3+E_{13}-E_4$ of self-intersection
$(-2)$. 
In Figure \ref{10A}, the six positive components in $G_4$ are just the 
cycle of curves disjoint from $E_4$.
\end{prop}
\begin{proof}
From our description of fixed curves of the projection as above, 
we see that $\sigma_4$ preserves all the curves $E_i$ and $E_{ij}$ except for $E_4$.
Compare Figures \ref{10A} and \ref{4D4}.

Consider the elliptic fibration $f\colon S\rightarrow \P^1$ defined by 
the divisor $2D_f=2(E_1+E_{12}+E_2+E_{23}+E_3+E_{13})$.
It gives the multiple fiber of type$\ _2{\rm I}_6$ in Kodaira's notation.
From Figure \ref{10A}, we see that the curve $E_4$ sits inside a reducible fiber
which we denote by $D'$.
In comparison with Figure \ref{branch}, $f$ corresponds to 
the pencil $\mathcal{L}$ 
of cubics on $\P^2$ spanned by the triangle $\{x_1x_2x_3=0\}$ and the cubic 
curve $C$ of \eqref{C}. Thus we see that the multiple fibers of $f$ are exactly the 
transform of the triangle, which is nothing but the divisor $2D_f$ of type$\ _2{\rm I}_6$, 
and the transform of $C$, namely some irreducible fiber of type$\ _2{\rm I}_0$.
On the other hand, the cubic 
\begin{equation}\label{Cinfty}
C_\infty: = L+Q\in \mathcal{L}
\end{equation}
corresponds to the 
reducible fiber of $f$ which contains $E_4$. Therefore the fiber $D'=E_4+B$ is of 
Dynkin type $\tilde{A}_1$ and is not multiple. (More precisely, 
it is of type III in characteristic three and otherwise ${\rm I}_2$.)
Since the Cremona involution of $\P^2$ interchanges $L$ and $Q$,
we see that $\sigma_4$ interchanges $E_4$ and $B$. 

From the linear equivalence 
$E_4+B\sim 2(E_1+E_{12}+E_2+E_{23}+E_3+E_{13})$, we 
see that the action is
\[\sigma_4\colon E_4\mapsto B= 2(E_1+E_{12}+E_2+E_{23}+E_3+E_{13})-E_4.\]
By taking the first paragraph into account, we see that $\sigma_4$ is 
numerically reflective and 
acts on $NS(S)_f$ by the reflection in the divisor
\[G_4=E_1+E_{12}+E_2+E_{23}+E_3+E_{13}-E_4.\]
\end{proof}
By symmetry, we obtain divisors $G_i (i=1,\dots, 4)$ which describe the numerically 
reflective involutions $\sigma_i$ in a similar manner.
We see that $(G_i,G_j)=2$ for $i\neq j$ so that the intersection diagram associated to divisors $G_1,\dots,G_4$ 
is the complete graph in four vertices with all edges doubled. 
In what follows we denote this configuration by $4C$.

We note that the automorphism $\sigma_i$ sends $G_i$ to its negative. It implies the 
following corollary.
\begin{cor}\label{G_i}
The numerical classes of $G_i$ are not effective.
\end{cor}
We can compute the intersections of $G_i$ and the $10A+6B$ configuration.
We have the following.
\begin{equation*}
\begin{split}
(G_i, E_j)=
\begin{cases}
2\quad \text{if $i=j$,}\\
0\quad \text{otherwise},
\end{cases}
(G_i, E_{kl})=0,\quad
(G_i,F_{kl})=
\begin{cases}
2\quad \text{if $i\not\in \{k,l\}$}\\
0\quad \text{if $i\in \{k,l\}$}.
\end{cases}
\end{split}
\end{equation*}

\begin{rmk}\label{limit}
The limit of our quartic surface $\overline{X}$ in \eqref{l=0} as $k \to \infty$ is the double  $\P^2$ with branch the union of the coordinate triangle and the reducible cubic $C_\infty$ in \eqref{Cinfty}.
Hence the limit of our Enriques surfaces is of type V in Kondo\cite{kondo86}.
(See \cite{M10} also.)
In this limit  our divisor class $G_4$  becomes effective and corresponds to the new singular point coming from the intersection $L \cap Q$  in \eqref{Cinfty}.
Thus $G_4$  can be regarded as the vanishing cycle of this specialization $k \to \infty$. 
Furthermore the numerically reflective involution $\sigma_4$ becomes numerically trivial in this limit.
More precisely, the limit of its graph as $k \to \infty$ is the union of that of the limit involution  and the product $C_4 \times C_4$, where  $C_4$  is the unique $(-2)$-curve representing  $G_4$ in the limit Enriques surface.
\end{rmk}

\section{Proof of the Theorems}\label{proof1}

In the previous two sections, we obtained sixteen smooth rational curves with the configuration
$10A+6B$ and four numerically reflective involutions $\sigma_i$ whose 
centers $G_i$ have the configuration $4C$. 
We begin with the consideration of the natural representation $r\colon \mathrm{Aut}(S)\rightarrow O(NS(S)_f)$.

\begin{prop}\label{NT}
The homomorphism $r$ is injective, namely there are no nontrivial 
numerically trivial automorphisms on $S$.
\end{prop}
\begin{proof}
Let $g$ be a numerically trivial automorphism of $S$, 
which is tame by virtue of Dolgachev\cite{D12}.
It preserves each $(-2)$-curves, in particular, each in the $10A+6B$ configuration. 
The curves $E_1,\dots,E_4$ in Figure \ref{10A} must be pointwise fixed,
since $\mathrm{Aut}(\P^1)$ is sharply triply transitive and 
since each $E_i$  has three distinct intersections with its neighbors.

We again focus on the elliptic fibration $f\colon S\rightarrow \P^1$ defined by 
$D_f=E_1+E_{12}+E_2+E_{23}+E_3+E_{13}$ as in Proposition \ref{NR}.
We saw that $E_4+B=E_4+\sigma_4(E_4)$ is a non-multiple fiber of $f$. 
Therefore the bisections $E_{14},E_{24},E_{34}$ of $f$ must intersect $B$. 
By a suitable choice of a bisection 
$C_f\in \{E_{14},E_{24},E_{34}\}$, we can assume that $C_f$ does not 
pass through the intersection $E_{4}\cap B$.
Then, since $g$ preserves all $(-2)$-curves, the curve $C_f$ has three distinct fixed points $E_4\cap C_f$, $B\cap C_f$ 
and $E_i\cap C_f$, where $E_i$ is the another vertex of the edge containing $C_f$
in Figure \ref{10A}. 
It follows that $g$ fixes $C_f$ pointwise, hence the singular curve $E_4+C_f$, too.
It follows that $g=\mathrm{id}_S$, since $g$ is tame and of finite order.
\end{proof}

In what follows 
we denote the hyperbolic lattice $NS(S)_f$ by $L$. 
Let us denote by $O'(L)$ the group of integral isometries 
whose $\R$-extensions preserve the positive cone of $L\otimes \R$.
We denote by $\Lambda$ the $9$-dimensional Lobachevsky space 
associated to the positive cone. 
Then $O'(L)$ acts on $\Lambda$ as a discrete group of motions.
We refer the readers to 
\cite{some} for the theory of discrete groups generated by reflections acting on 
Lobachevsky spaces. 

We let 
\[P^c =\{\R_+x\in PS(L)\mid (x,E)\geq 0 \text{ for all }E\in\{E_i,E_{ij},F_{ij},G_i\}\} \]
be the convex polyhedron defined by the twenty roots from the $10A+6B+4C$ configuration 
in the projective sphere 
$PS(L)=(L-\{0\})/\R_+$ (see \cite[Section 2]{some}). 
We have seen that every intersection number of distinct two divisors in $10A+6B+4C$
is in between $0$ and $2$,
hence the Coxeter diagram associated to these twenty roots has no 
dotted lines nor Lanner's subdiagrams. Also by an easy check
of the $10A+6B+4C$ configuration, we have the following.
\begin{lem}\label{maximal}
The Coxeter diagram of the polyhedron $P=P^c\cap \Lambda$ has exactly $29$ parabolic subdiagrams of maximal rank $8$.
They are as follows.
\begin{longtable}{c|c|c|c|c|c}
    & the subdiagram & number & $10A$ & $6B$ & $4C$ \\ \hline
$1)$ & $\tilde{E_7}+\tilde{A_1}$ & $12$ & $8$ & $1$ & $1$ \\
$2)$ & $\tilde{E_6}+\tilde{A_2}$ & $4$ & $7$ & $3$ & $0$ \\
$3)$ & $\tilde{D_6}+\tilde{A_1}+\tilde{A_1}$ & $6$ & $8$ & $1$ & $2$ \\
$4)$ & $\tilde{A_7}+\tilde{A_1}$ & $3$ & $8$ & $2$ & $0$ \\
$5)$ & $\tilde{A_5}+\tilde{A_2}+\tilde{A_1}$ & $4$ & $7$ & $3$ & $1$  
\end{longtable}
Here the each column $10A,6B,4C$ shows the number of vertices used from the 
configuration.
\end{lem}

It is easy to check that every connected parabolic subdiagram is a connected component 
of some parabolic subdiagram of rank $8$, using the previous table.
By Theorem 2.6 bis of \cite{some},
we see that $P$ has finite volume and we obtain $P^c\subset \overline{\Lambda}$.
This polyhedron gives the fundamental domain of the associated 
discrete reflection group generated by 
twenty reflections in the twenty roots
$\{E_i,E_{ij}, F_{ij},G_{i}\}$, which we denote by $W=W(10A+6B+4C)$. 
Algebro-geometrically, sixteen of the generators are the 
Picard-Lefschetz transformations in $(-2)$-curves in the $10A+6B$ configuration
and the rest four are the involutions $\sigma_i$ $(i=1,\dots,4)$ corresponding to $4C$.
As an abstract group, we see that $W$ has the structure of a Coxeter group whose fundamental relations are 
given by the Coxeter diagram (see \cite{some}) of $P$. 
We note that 
the quasi-polarization (namely a nef and big divisor) 
\[H=\sum_{i}E_i+\sum_{i<j}E_{ij}\] 
defines an element $\R_+ H$ in $P$.

Now let $W(4C)$ be the subgroup of $W$ generated by four reflections in $G_i$. 
Via the homomorphism $r\colon \mathrm{Aut}(S)\rightarrow O(NS(S)_f)$, 
the subgroup $\mathcal{F}\subset \mathrm{Aut}(S)$ generated by the four numerically reflective 
involutions $\sigma_i$ is mapped onto this Coxeter subgroup $W(4C)\simeq C_2^{*4}$. 
It follows that $\F\simeq W(4C)$. 
Let $W(10A+6B)$ be the subgroup generated by sixteen reflections in $E_i, E_{ij}$ and $F_{ij}$
and let $\overline{N}(W(10A+6B))$ be the minimal normal subgroup of 
$W$ which contains $W(10A+6B)$.
Since the intersection numbers between elements of $4C$ and $10A+6B$ are all even, by 
\cite[Proposition]{Vinberg}, we have the exact sequence
\begin{equation}\label{seq}
\begin{CD}
1 @>>> \overline{N}(W(10A+6B)) @>>> W @>>> W(4C) @>>> 1.
\end{CD}
\end{equation}
The kernel is exactly the subgroup generated by the conjugates 
\[\{\sigma g \sigma^{-1}\mid \sigma\in W(4C), g\text{ a generator of }W(10A+6B))\}.\]
We have the corresponding geometric consequence as follows.
\begin{thm}\label{rat}
There are exactly sixteen smooth rational curves on $S$ up to the action of $\F$.
\end{thm}
\begin{proof}
Let $E$ be a smooth rational curve on $S$. We consider the orbit $\F.E$.
Since the divisor $H$ above is nef, we can choose $E_0\in \F.E$ such that the degree $(E_0, H)$ is minimal.
We shall show that $E_0$ is one of sixteen curves in $10A+6B$. 

In fact, by the automorphism $\sigma_i$, we have 
\begin{equation*}
(E_0,H)\leq (\sigma_i(E_0),H)=(E_0,H)+(E_0,G_i)(G_i,H),
\end{equation*}
and $0\leq (E_0,G_i)$
for all $i$. Suppose that $E_0$ intersects non-negatively to 
all the sixteen curves in $10A+6B$. Then we have $\R_+ E_0\in P^c$. But from 
$P^c\subset \overline{\Lambda}$, we obtain
$(E_0^2)\geq 0$, which is a contradiction. 
Hence $E_0$ is negative on some curve in $10A+6B$ and we see that $E_0$ 
is one of them.

Next let us show that two distinct curves $E,E'$ in the $10A+6B$ configuration are 
inequivalent under $\F$. For the six curves $E_{ij}$ from $10A$, we have $(E_{ij},G_k)=0$ 
for all $i<j$ and $k$. Therefore, by an easy induction, we see that the sextuple 
$((E_{ij},E))_{1\leq i<j\leq 4}$ consisting of intersection numbers 
is an invariant of the orbit $\F.E$. 
Suppose that $(E_{ij},E)=(E_{ij},E')$ for all $i<j$. Since $E$ and $E'$ both are in 
the $10A+6B$ configuration, we see easily that $E=E'$.
This shows that the orbits $\F.E$ and $\F.E'$ are disjoint.
\end{proof}
In another words, the group $\overline{N}(W(10A+6B))$ is nothing but the 
Weyl group of $S$ generated by Picard-Lefschetz reflections in all $(-2)$-curves.
We can proceed to elliptic pencils.
\begin{thm}\label{ellip}
There are exactly $29$ elliptic pencils on $S$ up to the action of $\F$. Their properties 
are as in the table of Theorem \ref{elliptic}.
\end{thm}
\begin{proof}
Let $2f$ be a fiber class of an elliptic pencil on $S$. 
As before, we choose an element $f_0\in \F.f$ such that the degree $(f_0,H)$ is minimal.
We have
\begin{equation*}
(f_0,H)\leq (\sigma_i(f_0),H)=(f_0,H)+(f_0,G_i)(G_i,H),
\end{equation*}
hence $(f_0,G_i)\geq 0$ for all $i$. Moreover since $f_0$ is nef we have $(f_0,E)\geq 0$
for all $E$ in the $10A+6B$ configuration.
Therefore $f_0\in P^c$. This shows that $f_0$ 
corresponds to one of maximal parabolic subdiagrams classified in Lemma \ref{maximal}.

Conversely, we can construct $29$ elliptic pencils from the $29$ subdiagrams in Lemma~\ref{maximal} as follows.
The two types 2) and 4) in the lemma are easiest since they do not contain a class in $4C$.
The elliptic pencils of types 2) and 4) have singular fibers of type $\tilde{E_6}+\tilde{A_2}$ and $\tilde{A_7}+\tilde{A_1}$, respectively as in the case of \cite[Table 2]{kondo86}.

In the case of type 1) (resp. 5)), one component of the parabolic subdiagram is an $\tilde{A_1}$ consisting of a $(-2)$-curve $E$ in  $6B$  (resp. 10A)  and $(-2)$-class $G$  in  $4C$.
Moreover, the sum $E+G$ is a half of $\tilde{E_7}$  (resp. $\tilde{A_2}$).
Hence $E+\sigma(E)$  is a non-multiple fiber of type $\tilde{A_1}$ since  it is linearly equivalent to $2(E+G)$, where  $\sigma$ is the reflection in  $G$.

In the case of type 3), one component is an $\tilde{A_1}$ consisting of two classes  $G$  and  $G'$  in  $4C$.
But the other two components consisting of $(-2)$-curves.
Therefore, the Mordell-Weil group is of rank one since neither $G$ or  $G'$ is effective.
(The composite $\sigma\sigma'$ of two reflections in $G$  and  $G'$ is the translation by a generator of the Mordell-Weil group.)

That these $29$ pencils are inequivalent under $\F$ follows from the previous result
for rational curves.
\end{proof}
To study the image of the representation $r\colon \mathrm{Aut}(S)\rightarrow O(NS(S)_f)$,
we need some lemma.
We denote by $4A'$ the set of four roots $\{E_i\}$ and by $6A''$ the set $\{E_{ij}\}$.  
Recall that by Theorem \ref{rat}, all $(-2)$-curves on $S$ are in the $\mathcal{F}$-orbit of the three
sets $4A', 6A''$ and $6B$.
\begin{lem}
Let $\tau$ be any automorphism of  $S$.
Then $\tau$ preserves each of the three orbits of rational curves $\F.(4A')$,
$\F.(6A'')$ and $\F.(6B)$.
\end{lem}
\begin{proof}
Any automorphism $\tau$ permutes smooth rational curves on $S$,
hence induces a symmetry of the dual graph 
of the set of rational curves.
Thus, for the proof, it suffices to give a characterization of each orbit 
in terms of this infinite graph. 
We use the (full) subgraphs which are isomorphic to the 
dual graph of reducible fibers of elliptic fibrations.

Consider a vertex $v$ in $\F.(6B)$. Then there exists a subgraph of fiber type ${\rm I}_3$ 
passing through $v$. Conversely, if for a vertex $v$ there is a subgraph of fiber type ${\rm I}_3$, 
by Theorem \ref{ellip}, it is equivalent to a vertex in $6B$ under $\F$. Thus 
the vertices in $\F.(6B)$ are characterized by the property that there exists 
a subgraph of fiber type ${\rm I}_3$ passing through them. 

Similarly, vertices in $\F.(10A)$ are characterized by subgraphs of type ${\rm I}_8$. 
Moreover, the vertices $v$ in $\F.(4A')$ are characterized 
by the property that there exists a subgraph of type ${\rm IV}^*$ which has 
$v$ as its end. In the opposite way, vertices in $\F.(6A'')$ are those which 
does not have such ${\rm IV}^*$ subgraphs. Thus the three orbits are all characterized 
and $\tau$ preserves these orbits.
\end{proof}
\begin{cor}
The set of six curves $\{E_{ij}\}$ is preserved under any automorphism.
\end{cor}
\begin{proof}
In fact, for any $E_{kl}$ and any $\sigma_i$ we have $\sigma_i(E_{kl})=E_{kl}$. Hence 
$\mathcal{F}.(6A'')=\{E_{ij}\}$.
\end{proof} 
Recall that $S$ has the action by $\mathfrak{S}_4$ from the symmetry of the 
defining equation of $\overline{X}$. Explicitly, it acts on the curves in $10A+6B$ 
configuration by the permutation of indices. For involutions $\sigma_i$, the same 
holds true if we regard the action as taking conjugates. 
It is easy to see that this group $\mathfrak{S}_4$ can be 
identified with the symmetry group $\mathrm{Sym}(P)$ of the 
polyhedron $P\subset \Lambda$ via $r$. We can also regard this group as acting on the 
reflection group $W$ and the exact sequence \eqref{seq} is preserved under this action.
In particular, $W(4C)$ and $\mathrm{Sym}(P)$ generate a group isomorphic to 
$\mathfrak{S}_4\ltimes C_2^{*4}$.
\begin{thm}\label{mt}
The representation $r$ induces an isomorphism of $\mathrm{Aut}(S)$ onto the group generated by $W(4C)$ and $\mathrm{Sym}(P)$, hence we obtain $\mathrm{Aut}(S)\simeq 
\mathfrak{S}_4\ltimes C_2^{*4}$.
\end{thm}
\begin{proof}
Since $r$ maps $\F$ onto $W(4C)$, the image of $r$ includes the groups $W(4C)$ 
and $\mathrm{Sym}(P)$.

Conversely let us pick up an arbitrary automorphism $\tau$.
We consider the image $\tau (H)$ of $H$. 
We use the elliptic fibration defined by the divisor 
$f=H-E_{12}-E_{34}$ of type ${\rm I}_8$. 
By Theorem \ref{ellip}, the image $\tau (f)$ 
is equivalent to 
one of three elliptic pencils described in item $4)$ under $\F$. 
Moreover, since 
$\mathfrak{S}_4$ acts transitively on these three pencils, 
we can assume that $\tau (f)=f$ by composing $\tau$ with some elements 
of $\F$ and $\mathfrak{S}_4$.
Thus we have $\tau (H)=f+\tau (E_{12})+\tau(E_{34})$. 
By the previous corollary, $\tau (E_{12})$ and $\tau(E_{34})$ are in the set $\{E_{ij}\}$.
By an easy check of intersection numbers, we see that $\tau (E_{12})+\tau(E_{34})=
E_{12}+E_{34}$. In particular we obtain $\tau (H)=H$ as divisors.
Since any permutation of the $10A$ configuration can be induced from 
the automorphism group $\mathfrak{S}_4$, this shows that the image of $r$ is 
contained in the group generated by $W(4C)$ 
and $\mathrm{Sym}(P)$.
\end{proof}


\begin{thebibliography}{99}

\bibitem{BP} 
	{W. Barth and C. Peters},
	{Automorphisms of Enriques surfaces}, 
	{Invent. Math.},
	\textbf{73} (1983), 383--411.

\bibitem{BS}
	{A. Borel and J. P. Serre},
	{Corners and arithmetic groups}, 
	{Comment. Math. Helv.},
	\textbf{48} (1973), 436--491.

\bibitem{D}
	{I. Dolgachev},
	{On automorphisms of Enriques surfaces},
	{Invent. Math.},
	\textbf{76} (1984), 163--177.

\bibitem{D12}
	{I. Dolgachev},
	{Numerical trivial automorphisms of Enriques surfaces in arbitrary characteristic},
	in {\it Arithmetic and Geometry of $K3$ surfaces and Calabi-Yau threefolds}, 
	Fields Institute Communications {\bf 67}, 2013, pp. 267--283.

\bibitem{kondo86} 
	{S. Kondo},
	{Enriques surfaces with finite automorphism groups},
	{Japan. J. Math.},
	\textbf{12} (1986), 191--282.

\bibitem{M10}
	{S. Mukai},
	{Numerically trivial involutions of Kummer type of an Enriques surface},
	{Kyoto J. Math.},
	\textbf{50} (2010), 889--902.

\bibitem{mukai-ref}
	{S. Mukai},
	{Kummer's quartics and numerically reflective involutions of Enriques surfaces},
	{J. Math. Soc. Japan},
	\textbf{64} (2012), 231--246. 

\bibitem{Fields}
	{S. Mukai and H. Ohashi},
	{Enriques surfaces of Hutchinson-G\"opel type and Mathieu automorphisms},
	in {\it Arithmetic and Geometry of $K3$ surfaces and Calabi-Yau threefolds}, 
	Fields Institute Communications {\bf 67}, 2013, pp. 429--454.

\bibitem{N}
	{V. V. Nikulin},
	{On a description of the automorphism groups of Enriques surfaces},
	{Soviet Math. Dokl.},
	\textbf{30} (1984), 282--285.

\bibitem{Serre}
	{J. P. Serre},
	{Cohomologie des groupes discrets},
	{Prospects in mathematics (Proc. Symp., Princeton Univ., Princeton, N.J., 1970)},
	pp. 77--169. 
	{Ann. of Math. Studies, No. 70, Princeton Univ. Press, Princeton, N.J., 1971.}

\bibitem{Vinberg}
	{E. B. Vinberg},
	{The two most algebraic $K3$ surfaces},
	{Math. Ann.},
	\textbf{265} (1983), 1--21.

\bibitem{some}
	{E. B. Vinberg},
	{Some arithmetical discrete groups in Loba\^{c}evski\^{i} spaces},
	      in {\it Discrete Subgroups of Lie Groups and Appl. to Moduli (Bombay 1973)},
	Oxford University Press, 1975, pp. 323--348.
\end{thebibliography}
\end{document}